\documentclass[12pt]{article}

\usepackage{amsmath,amsthm}
\usepackage{graphicx}
\usepackage{color}
\usepackage{appendix}
\usepackage{mathrsfs}

\newtheorem{theorem}{Theorem}[section]

\newtheorem{remark}{Remark}[section]

\newtheorem{lemma}{Lemma}[section]

\allowdisplaybreaks

\begin{document}

\title{\bf Approximations for a Queueing Game Model with Join-the-Shortest-Queue Strategy}
\author{Qihui Bu$^1$, Liwei Liu$^2$, Jiashan Tang$^{1*}$ and Yiqiang Q. Zhao$^3$\\
{\em\small 1: College of Science, Nanjing University of Posts and Telecommunications}\\
{\em\small Nanjing 210023, Jiangsu, China}\\
{\em\small 2: School of Science, Nanjing University of Science and Technology}\\
{\em\small Nanjing 210094, Jiangsu, China}\\
{\em\small 3: School of Mathematics and Statistics, Carleton University}\\
{\em\small Ottawa, Ontario Canada  K1S 5B6}\\
{\em\small\em $^{*}$Corresponding Author: tangjs@njupt.edu.cn}}
\date{}
\maketitle

\renewcommand{\baselinestretch}{1.3}
\large\normalsize
\begin{abstract}
This paper investigates a partially observable queueing system with $N$ nodes in which each node has a dedicated arrival stream. There is an extra arrival stream to balance the load of the system by routing its customers to the shortest queue. In addition, a reward-cost structure is considered
to analyze customers' strategic behaviours. The equilibrium and socially optimal strategies are derived for the partially observable mean field limit model. Then, we show that the strategies obtained from the mean field model are good approximations to the model with finite $N$ nodes. Finally, numerical experiments are provided to compare the equilibrium and socially optimal behaviours, including joining probabilities and social benefits for different system parameters.
\end{abstract}

\section{Introduction}
Parallel queueing systems have a wide range of practical applications, such as computer networks, cognitive radio networks, transport systems. Such parallel queueing systems become more large-scaled and more common, largely due to the rapid advances of technologies, broad involvements of big data, and popularity of computer and smart phone applications. It is popular and reasonable to improve system performances by various load balancing schemes. The join-the-shortest-queue (JSQ) scheme is one of the most frequently used balancing mechanisms, in which customers make assignment decisions according to the state information about all servers in the system. Haight~\cite{FH} and Kingman~\cite{JK} first studied JSQ scheme in a queueing system with two parallel queues and derived the stationary distribution for the queue length for each queue. Until now, there is a growing number of papers that deal with queueing system with the join-the-shortest-queue scheme, see McDonald~\cite{DM} and Foley and McDonald~\cite{FM} for systems with several servers. When the number of servers is very large, readers can refer to Dawson, Tang and Zhao~\cite{DTZ}, Gupta, Balter {\it et al.}
~\cite{VG}, Bramson, Lu and  Prabhakar~\cite{MB}, Guillemin, Olivier {\it et al.}~\cite{FG} and Eschenfeldt and Gamarnik~\cite{PE}, among others.

During the last two decades, considerable attention has been devoted to the game-theoretic analysis of queueing systems. Reward-cost structure is introduced in queueing systems to discuss customers' strategic behaviours, balking or joining. According to the levels of information of the system on hand, customers maximize their expected utilities to make their decisions. Such an economic analysis of queueing system was pioneered by Naor~\cite{PN}. He investigated a single-server system with an observable queue, in which an arriving customer can observe the number of customers and then decide to join or balk. Edelson and Hildebrand~\cite{EH} studied the same queueing system and assumed that arriving customers are not informed of the queue length of the system.
Recent works on this topic include Chen and Zhou~\cite{CZ}, Hassin and Roet-Green~\cite{HR}, Do, Van Do and Melikov~\cite{DVM} and Wang {\it et al.}~\cite{WLS}. Interested readers can refer to the book of Hassin~\cite{RH} for more details.

Our study is inspired by Dawson, Tang and Zhao's work~\cite{DTZ}. They discussed a queueing network with $N$ nodes, in which $N$ is large and each queue has a dedicated input stream. In addition, there is an extra input stream, which balances the load of the network by directing its arriving customers to the shortest queue. The first contribution of 
the current paper is to introduce the reward-cost structure to Dawson, Tang and Zhao's work. We assume that customers are strategic and decide to join or balk the system according to their expected utilities. The second contribution is to provide an alternative
method to get the result, obtained in ~\cite{DTZ}, that the stationary behaviour of any queue in the parallel queueing network can be approximated by that of an $M/M/1$ queue with a modified arrival rate. Applying the Lyapunov function, we prove that the fixed point, which is the stationary distribution of the modified $M/M/1$ queue, is globally asymptotically stable. Furthermore, we justify the interchange of the limits, $N\to\infty$ and $t\to\infty$, for the finite $N$ model and verify that the limiting result is a one-point distribution about the fixed point. 
Finally, we discuss game theory in the mean field limit model rather than the actual discrete model and justify that the equilibrium strategy in the mean field limit model is an $\epsilon$-Nash equilibrium in the actual discrete model. There are very few literature about using approximation method to analyze queueing game. Xu and Hajek~\cite{XH} proposed a supermarket game in which customers decide how many queues to sample to minimize their waiting cost. Wiecek, Altman and Ghosh~\cite{WAG} proposed an admission control of an $M/M/\infty$ queueing system with shared service cost and considered that a mean field game approximates the game in the discrete model.

The organization of the paper is as follows: We review the model settings and some notation in~\cite{DTZ} and introduce a reward-cost structure in Section 
\ref{Sec:2}. In Section 
\ref{Sec:3}, we determine the Nash equilibrium and socially optimal strategies for the mean field model in the partially observable case. In Section 
\ref{Sec:4}, we obtain the approximation of the stationary distribution for the finite $N$ model, which was provided in Theorem $1$ in ~\cite{DTZ}. We prove that the fixed point is globally asymptotically stable and when $N\to\infty$, the stationary state of the finite $N$ model converges to the fixed point. In addition, we establish the convergence of the game to its mean field limit model. In Section 
\ref{Sec:5}, some numerical experiments are provided to show the impacts of system parameters on the equilibrium, 
optimal strategies and social benefits.

\section{Model description}\label{Sec:2}

In this section, we first review the model with finite $N$ nodes (referred as a finite $N$ model in the sequel) 
in~\cite{DTZ} and some notation. Meanwhile, reward and waiting cost are introduced into the finite model for further analysing customers' strategic behaviours.

A queueing system with $N$ parallel and identical queues are considered in that 
paper, where each queue has an infinite waiting space and a server with exponential service rate $\mu$. Each queue is equipped with a dedicated customer arrival stream, which follows a Poisson process with parameter $\lambda$. In addition, there is an extra Poisson arrival stream with the 
arrival rate $N\lambda_s$, and the (smart) customers from this stream join the shortest queue among $N$ queues. We conduct our following analysis under the stationary condition $\lambda_s+\lambda<\mu$.

Let $X^N_i(t)$ denote the queue length of the $i$-th queue at time $t$, $i=
1,2,\ldots,N$. For each nonnegative integer $k$, let $U^N_k(t)$ be the fraction of the number of severs with exactly $k$ customers at time $t$,
\begin{equation}
  U^N_k(t)= \frac 1{N} {\sum\limits_{i=1}\limits^{N}\delta_{k}(X^N_i(t))},
\end{equation}
where $\delta_{k}(x)=\left\{\begin{array}{cc}
                                  1,& x=k, \\
                                 0,& x\neq k.
                               \end{array}
\right.$
Then, it is clear that $\{(U^N_0(t), U^N_1(t),U^N_2(t),\ldots), t \ge 0\}$ is an infinite-dimensional continuous time Markov process with state space $\{(u^N_0,u^N_1,u^N_2,\ldots):0\leq u^N_k\leq 1, Nu^N_k\in {\bf N}, k=0,1,2,\ldots, \sum\limits^{\infty}\limits_{k=1}u^N_k=1\}$.

If the arrival customer joins a queue with $k$ customers, then the number of queues with $k$ customers is reduced by one and the number of queues with $k+1$ customers is increased by one. The state of the system changes from $N u^N$ to $N u^N-{\bf e}_{k}+{\bf e}_{k+1}$ and the corresponding arrival rate is $ Nu^N_k\lambda+N\lambda_s\delta_{k}(\min\{k':u^N_{k'}>0\})$, where ${\bf e}_{k}$ denotes a vector whose $k$-th entry is 1 and 0 elsewhere. Similarly, we can obtain other non-zero changing rates on the number of queues with $k$ customers for state $u^N=(u^N_0,u^N_1,u^N_2,\ldots)$, or equivalently $N u^N=N(u^N_0,u^N_1,u^N_2,\ldots)$, as follows:
\[\begin{array}{l}
   N u^N \to N u^N-{\bf e}_{k}+{\bf e}_{k-1}:~~N\mu u^N_k (1-\delta_0(k)), 
   \\[4mm]
   N u^N \to N u^N-{\bf e}_{k-1}+{\bf e}_{k}:~~Nu^N_{k-1}\lambda+N\lambda_s  \delta_{k-1}(\min\{k':u^N_{k'}>0\}), 
   \\[4mm]
    N u^N \to N u^N-{\bf e}_{k+1}+{\bf e}_{k}:~~ N\mu u^N_{k+1}.
\end{array}\]

In addition, the mean field limit is used to approximate the stationary performance of the above finite $N$ model with interactions among queues in terms of the limiting dynamics, which arises in the limit as $N$ goes to infinity. According to Theorem $1$ in~\cite{DTZ}, the limit of the finite $N$ model is the solution of a deterministic system and the corresponding queueing system is called the mean field limit model, which is referred to as the limiting typical queue in~\cite{DTZ}. 
All notation in the mean field limit model are consistent with those in the finite model without superscript $N$.

We further assume in this paper that customers are strategic and can choose either to join or balk upon arrivals, based on their expected utilities.  Each customer receives a reward $R$ units after its service. Additionally, customers suffer a waiting cost $C$ units per time unit in the system. If a customer joins the system, its expected utility is $R-CE(W)$, where $E(W)$ is the average sojourn time of the customer in the system, and the expected utility of a customer who balks is defined to be zero. Hence, with the objective of maximizing the expected utility, a new arrival customer decides to join the system when $R-CE(W)>0$, or to balk when $R-CE(W)<0$. It is indifferent for a customer to choose to join or balk when $R-CE(W)=0$. Throughout the paper, we assume $R\geq\frac{C}{\mu}$, which ensures that a customer must join the queue when the queue is empty. 

\section{Nash equilibrium in the mean field model}\label{Sec:3}

We consider the partially observable finite model, in which all dedicated customers are not informed about the number of customers in the system, while all smart customers from the extra arrival stream are informed about it.
It is worthwhile to note that when $N$ is large enough, there always exits an empty queue among the $N$ parallel queues in the steady state. Since in this case we always have $R-\frac{C}{\mu}>0$, the smart customers must enter the system by joining the shortest queue strategy. 

 In this section, we focus on the equilibrium and socially optimal strategies of the customers from the dedicated arrival streams, specifically, customers' strategic behaviours for maximizing customer's individual expected utility and customers joining probability for optimizing social benefits are studied. This model can be considered  as a symmetric game among indistinguishable customers. We use $q$ to  denote the parameter of the strategy of the customers. Specifically, 
$q\in [ 0,1]$ denotes the probability of joining.
Obviously, $q=0$ denotes balking and $q=1$ denotes joining.
The mean field model with game behaves as the original mean field model, but with dedicated arrival rate $\lambda q$ instead of $\lambda$.

\begin{theorem}
The mean sojourn time of a customer from the dedicated arrival stream in the mean field model is given by
\begin{equation} 
  E(W(q))=\displaystyle\frac{\mu+\lambda_s}{\mu}\times\frac{1}{\mu(1-\sigma)}.
\end{equation}

\end{theorem}

\begin{proof}
According to Theorem $2$ 
in~\cite{DTZ}, the stationary distribution of the mean field model can be computed by,
\begin{eqnarray}
 \label{pi0} \pi_0 &=& 1-\rho, \\
 \label{pik} \pi_k &=& \rho(1-\sigma)\sigma^{k-1},
\end{eqnarray}
where $\rho=(\lambda q+\lambda_s)/\mu$ and $\sigma=\lambda q/\mu$.

Then, the mean sojourn time of a customer from the dedicated arrival stream is given by:
\begin{eqnarray}
 \nonumber E(W(q))&=& \sum\limits_{k=0}\limits^{\infty}\frac{k+1}{\mu}\pi_k\\
\nonumber &=&\frac{1}{\mu}(1-\rho)+\sum\limits_{k=1}\limits^{\infty}\frac{k+1}{\mu}\rho(1-\sigma)\sigma^{k-1}\\
 \nonumber   &=&  \displaystyle\frac{\mu+\lambda_s}{\mu}\times\frac{1}{\mu(1-\sigma)}.
\end{eqnarray}
\end{proof}

\begin{remark}
For conveniences, in the limiting model, we divide the arrival stream into the normal and the external arrival streams with parameters $\lambda q$ and $\lambda_s/\pi_0$, respectively. Note that the normal stream always visits the system while the external stream arrives to the system only when the system is empty. The probability that an arbitrary customer is from the normal stream is given by,
\begin{equation*}
  \frac{\lambda q}{\lambda q+\lambda_s} \sum\limits_{k=0}\limits^{\infty}\pi_k =\frac{\lambda q}{\lambda q+\lambda_s}.
\end{equation*}
The probability that an arbitrary customer is from the external stream is given by,
\begin{equation*}
  \frac{\lambda_s/\pi_0 \times\pi_0}{\lambda q+\lambda_s}=\frac{\lambda_s}{\lambda q+\lambda_s}.
\end{equation*}
Thus, the mean sojourn time for an arbitrary customer is
\begin{equation*}
  E(H(q))=\frac{\lambda q}{\lambda q+\lambda_s}E(W(q))+\frac{\lambda_s}{\lambda q+\lambda_s}\frac{1}{\mu}.
\end{equation*}
The expected queue length is given as follows,
\begin{eqnarray*} 
  E(L(q))&=&\sum\limits_{k=0}\limits^{\infty}k\pi_k=\sum\limits_{k=1}\limits^{\infty}k\rho(1-\sigma)\sigma^{k-1}\\
  &=&\frac{\rho}{1-\sigma}.
\end{eqnarray*}

We obtain the following relationship about the mean sojourn time, expected queue length and effective arrival rate, 
\begin{equation*}
  E(H(q))=\frac{E(L(q))}{\lambda q+\lambda_s},
\end{equation*}
which is the Little's Law for the queueing systems studied in this paper.
\end{remark}


Based on the expected sojourn time and the expression of customer's expected utility, customers' behaviours can be investigated as follows.

\begin{theorem}
Considering the Nash equilibrium in the partially observable mean field model, we have the following equilibrium strategy for the dedicated customers:
\begin{equation}\label{eq strategy}
  q^e=\left\{\begin{array}{ll}
               0, & \text{if } 0<\displaystyle\frac{R}{C}\leq \displaystyle\frac{\lambda_s+\mu}{\mu^2}, \\[4mm]
\displaystyle\frac{\mu}{\lambda}-\frac{C(\mu+\lambda_s)}{R\mu\lambda}, & \text{if } \displaystyle \frac{\lambda_s+\mu}{\mu^2}<\frac{R}{C}< \frac{\mu+\lambda_s}{\mu(\mu-\lambda)},\\[4mm]
               1,& \text{if } \displaystyle \frac{R}{C} \geq \frac{\mu+\lambda_s}{\mu(\mu-\lambda)}.
             \end{array}
  \right.
\end{equation}

\end{theorem}

\begin{proof}
The derivative of the expected sojourn time with respect to $q$ is given by
\begin{equation*} 
 \frac{{\rm d}E(W(q))}{{\rm d}q}=\frac{\lambda (\mu+\lambda_s)}{\mu(\mu-\lambda q)^2}>0,
\end{equation*}
which implies that $E(W(q))$ is strictly increasing for $q\in[0,1]$. The expected utility of a customer from a dedicated arrival stream, $S(q)=R-CE(W)$, is strictly decreasing for $q\in[0,1]$ and has a unique maximum at
\begin{equation*} 
 S(0)=R-C\frac{\lambda_s+\mu}{\mu^2},
\end{equation*}
and a unique minimum at
\begin{equation*} 
   S(1)=R-C\frac{\mu+\lambda_s}{\mu(\mu-\lambda)}.
\end{equation*}

\begin{description}
\item[Case 1: ] When $0<\displaystyle\frac{R}{C}\leq \displaystyle\frac{\lambda_s+\mu}{\mu^2}$, $S(0)\leq0$ and $S(q)\leq0$ for all $q$. Then, the best response for the customer is balking and thus the equilibrium strategy is $q^e=0$, which is the first branch of \eqref{eq strategy}.

\item[Case 2: ] When $\displaystyle \frac{\lambda_s+\mu}{\mu^2}<\frac{R}{C}< \frac{\mu+\lambda_s}{\mu(\mu-\lambda)}$, $S(0)>0$ and $S(1)<0$. $q_1=\displaystyle\frac{\mu}{\lambda}-\frac{C(\mu+\lambda_s)}{R\mu\lambda}$ is the unique solution to $S(q)=0$. When $q\in[0,q_1)$, the best response for the customer is joining.
When $q\in(q_1,1]$, balking is the best response. When $q=q_1$, all strategies are the best response. Hence, in this situation, $q^e=q_1$ is the equilibrium strategy, which is given in the second branch of \eqref{eq strategy}.

\item[Case 3: ] When $\displaystyle\frac{R}{C}\geq \frac{\mu+\lambda_s}{\mu(\mu-\lambda)}$, $S(1)\geq 0$ and $S(q)>0$ for every $q\in[0,1]$. Therefore, joining is the best response for the customer. Thus, $q^e=1$ is the unique Nash equilibrium strategy and the third branch of \eqref{eq strategy} is obtained.
\end{description}
\end{proof}

Next, we study the social welfare optimization issue which maximizes the system's overall benefits per time unit. Results are given in the following theorem.
\begin{theorem}
 In the partially observable mean field model, there is a unique mixed strategy, the customer from dedicated arrival stream joining the system with probability $q^*$, which maximizes the social benefits per time unit exists, provided by,
\begin{equation*}
  q^*=\left\{\begin{array}{ll}
               0, & R\leq \displaystyle C\frac{\lambda_s+\mu}{\mu^2}, \\
               \displaystyle \frac{\mu}{\lambda}-\frac{1}{\lambda}\sqrt{\frac{C(\lambda_s+\mu)}{R}}, & \displaystyle C\frac{\lambda_s+\mu}{\mu^2}<R<C\frac{\lambda_s+\mu}{(\mu-\lambda)^2}, \\
               1, &\displaystyle R\geq C\frac{\lambda_s+\mu}{(\mu-\lambda)^2}.
             \end{array}
\right.
\end{equation*}
\end{theorem}

\begin{proof}
When all customers from dedicated arrival stream adopt strategy $q$, i.e. joining the system with probability $q$, the social benefits per time unit is as follows:
\begin{equation*} 
  S^{un}_{soc}(q) = (\lambda q+\lambda_s)R-CE(L(q))
= R\left(\lambda q+\lambda_s\right)-C\frac{\lambda q+\lambda_s}{\mu-\lambda q}.
\end{equation*}
Successive differentiations of $S^{un}_{soc}$ gives us
\begin{eqnarray*}
  \frac{{\rm d}S^{un}_{soc}(q)}{{\rm d}q}&=&R\lambda-C\frac{\lambda(\lambda_s+\mu)}{(\mu-\lambda q)^2},\\
  \frac{{\rm d}^2S^{un}_{soc}(q)}{{\rm d}q^2}&=&-\frac{2C\lambda^2(\lambda_s+\mu)}{(\mu-\lambda q)^3}.
\end{eqnarray*}
Under the stability condition of the model: $\lambda+\lambda_s<\mu$, which has been assumed throughout the paper, it is obvious that $\displaystyle \frac{{\rm d}^2S^{un}_{soc}(q)}{{\rm d}q^2}<0$ for $q\in[0,1]$ and that $S^{un}_{soc}(q)$ is a concave function. The solution to $ \displaystyle \frac{{\rm d}S^{un}_{soc}(q)}{{\rm d}q}=0$ is $\displaystyle q'=\frac{\mu}{\lambda}-\frac{1}{\lambda}\sqrt{\frac{C(\lambda_s+\mu)}{R}}$.
\begin{description}
  \item[Case 1: ] When $ \displaystyle R\leq C\frac{\lambda_s+\mu}{\mu^2}$, i.e. $q'\leq0$ and $\displaystyle \frac{{\rm d}S^{un}_{soc}(q)}{{\rm d}q}\leq0$, which implies that $S^{un}_{soc}(q)$ is strictly decreasing in $q$ and the unique maximum is attained at $q=0$. Thus, the best response is $q^*=0$;
  \item[Case 2: ]  When $ \displaystyle C\frac{\lambda_s+\mu}{\mu^2}<R<C\frac{\lambda_s+\mu}{(\mu-\lambda)^2}$, $0<q'<1$ and the unique maximum of $S^{un}_{soc}(q)$ is  attained at $q=q'$, i.e. $q^*=q'$;
  \item[Case 3: ]  When $ \displaystyle R\geq C\frac{\lambda_s+\mu}{(\mu-\lambda)^2}$, $\displaystyle \frac{{\rm d}S^{un}_{soc}(q)}{{\rm d}q}\geq0$. $S^{un}_{soc}(q)$ is strictly increasing in $q$ and the unique maximum is attained at $q=1$. Thus, the best response is $q^*=1$.
\end{description}
\end{proof}

\section{$\epsilon$-Nash Equilibrium for the original model}\label{Sec:4}

In this section, we prove that the Nash equilibrium in the mean field model is an $\epsilon$-Nash equilibrium in the finite $N$ model.
Before proceeding, it is worth noting that in Dawson, Tang and Zhao's work~\cite{DTZ}, for the model without game, they assumed that the initial state of process $\{U^N_k(t), k=0,1,2,\ldots\}$ must be $U^N_0(0)=1$ and under this condition the approximate stationary distribution, \eqref{pi0}--\eqref{pik}, is derived. In the following, we provide an alternative method to prove  $\lim\limits_{N\to\infty}\lim\limits_{t\to\infty}u^N(t,u(0))=\pi,$ for any initial state $u(0)$.

\begin{lemma}
From any initial state, the probability that the system (the typical queue) is empty at time $t$ ($t>0$) is larger than zero, i.e., $u_0(t)>0$.
\end{lemma}

\begin{proof}
We prove this lemma by the contradiction argument. For any initial state $u(0)$, suppose that $\exists$ $t_0>0$ such that $u_0(t_0)=0$.
Without 
loss of generality, let $u_k(0)$ be a nonzero element of $u(0)$, denoting the probability of the system with $k$ customers at time $0$.
Then, the transition probability $P_{u_k(0)\to u_0(t_0)}$ from $u_k(0)$ to $u_0(t_0)$ would be zero, which contradicts to the following:
\begin{eqnarray}
  \nonumber  P_{u_k(0)\to u_0(t_0)}&=&\sum\limits^{\infty}\limits_{i=0}P\left(\text{ there are $i$ new customers arrive at the system during $(0,t_0]$}\right)\\
  &\times &P\left(\text{there are $i+k$ customers served during $(0,t_0]$}\right)\\
  &=&\sum\limits^{\infty}\limits_{i=0}P\left(\sum\limits_{n=1}^{i}a_{n}<t_0,\sum\limits_{n=1}^{i+1}a_{n}>t_0\right)\times P\left(\sum\limits_{n=1}^{i+k}s_{n}<t_0\right)
>0,
\end{eqnarray}
where $s_n$ and $a_n$ denote the service time of the $n$-customer and the interarrival time between the $(n-1)$-st and the $n$-th customers, respectively.
\end{proof}

\begin{theorem}\label{gas}
\begin{equation*}
  \lim\limits_{t\to\infty}u(t,u(0))=\pi, \text{ for any initial state } u(0),
\end{equation*}
i.e. the fixed point $\pi$ of the limiting system is globally asymptotically stable.
\end{theorem}

\begin{proof}
Motivated by the proof provided in \cite{MM2}, we define the following Lyapunov function
\begin{equation*}
  V(t)=\sum\limits_{i=1}\limits^{\infty}w_i\mid\epsilon_{i}(t)\mid,
\end{equation*}
where $w_i>0$ and $\epsilon_{i}(t)=u_i(t)-\pi_i$ for $i=1,2,3,\ldots$.
Suppose that $t_0>0$. For any $t\in[t_0,\infty)$, according to Lemma 4.1 and
 $\displaystyle\frac{{\rm d}\epsilon_i(t)}{{\rm d}t}=\frac{{\rm d}u_i(t)}{{\rm d}t}$, the following differential equations can be obtained based on the results of~\cite{DTZ}
and by noticing that $\pi$ is the fixed point:
\begin{eqnarray}
\label{de1}\frac{{\rm d}\epsilon_{0}(t)}{{\rm d}t}&=&-\lambda\epsilon_0+\mu\epsilon_1,\\
\label{de2}\frac{{\rm d}\epsilon_{i}(t)}{{\rm d}t}&=&\lambda\epsilon_{i-1}-(\lambda+\mu)\epsilon_{i}+\mu\epsilon_{i+1}.
\end{eqnarray}

In the following analysis, we firstly assume that $\epsilon_i(t)\neq 0$ for all $i\ge 1$ and all $t\ge 0$, because when $\epsilon_i(t)= 0$ at  $t$ for some $i\ge 1$, the derivative $dV/dt$ of $V$ with respect to $t$ is not well defined, and this case will be treated at the end of the proof.

The derivative of $V$ with respect to $t$ can be written as
\begin{equation}\label{dv}
  \displaystyle\frac{dV}{dt}=\sum\limits_{i=0}\limits^{\infty}w_i\cdot \mbox{sgn}(\epsilon_{i}(t))\frac{{\rm d}(\epsilon_{i}(t))}{{\rm d}t},
\end{equation}
where $\displaystyle \mbox{sgn}(x)=\left\{\begin{array}{ll}
                1, & x>0,\\
                0, & x=0,\\
                -1, & x<0.
              \end{array}\right.$

Substitute the differential equations in \eqref{de1}--\eqref{de2} for those in \eqref{dv}, and collect all terms involving  $\epsilon_{i}^u(t)$ ($i=1,2,3,\ldots,$) to give
\begin{eqnarray}
 \nonumber\hspace{-3mm}T(\epsilon_{i}(t))&\equiv& w_{i-1}\text{sgn}(\epsilon_{i-1}(t))\mu\epsilon_i(t)-w_{i}\text{sgn}(\epsilon_{i}(t))(\lambda+\mu)\epsilon_i(t)\\&
+&w_{i+1} \text{sgn}(\epsilon_{i+1}(t))\lambda\epsilon_i(t).
\end{eqnarray}
 Note that $i=0$ is a special case, which can be included in the above form by taking $w_{-1}=0$.

%
%
%

 We can choose appropriate positive $w_{i-1}$, $w_i$ and $w_{i+1}$ (except for $w_{-1}=0$) such that $ T(\epsilon_{i}(t))$ are negative for all $i=0,1,2,\ldots$. Specifically, by setting $w_{0}=1$, we choose $w_{i+1}$, $i=1,2,\ldots$, such that the following inequalities hold for respective values of $i$:
\begin{eqnarray} 
 w_1&<&w_0,\\
 w_{i+1}&<&w_i+\frac{\mu}{\lambda}(w_i-w_{i-1}).
\end{eqnarray}
This can be done from $T(\epsilon_{i}(t))<0$ and when $\mbox{sgn}(\epsilon_{i-1}(t))=\mbox{sgn}(\epsilon_{i}(t))=\mbox{sgn}(\epsilon_{i+1}(t))=1$ or $-1$. Otherwise, it is easy to see that the value of $T(\epsilon_{i})$ is decreasing.

The above selections yield $\displaystyle\frac{{\rm d}V}{{\rm d}t}<0$  for $\epsilon \neq 0$. Applying Lyapunov Theorem, we see that $\lim\limits_{t\to\infty}u(t,u(t_0))=\pi$. Then, it follows that
\[\lim\limits_{t\to\infty}u(t,u(0))=\lim\limits_{(t-t_0)\to\infty}u(t-t_0,u(t_0))=\pi,\]
i.e. the fixed point $\pi$ of the limiting system is globally asymptotically stable.

Finally, we discuss the case, in which $ \epsilon_{i}(t)=0$ for some $i$ at some $t$. Since we only study the forward progress of the system, it is reasonable and sufficient to consider the right-hand derivatives of $\epsilon_{i}(t)$ (For example, see reference \cite{YSK}, \cite{MM2}). The above proof prevails for right-hand derivatives.
\end{proof}

By invoking Prokhorov's theorem and applying the above theorem, we can establish $\lim\limits_{N\to\infty}u^N(t,c)=\pi$, which is shown in the following theorem.
\begin{theorem}\label{main result}
\begin{equation*}
 \lim\limits_{N\to\infty}\Pi_{N}=\delta_{\pi},
\end{equation*}
or equivalently $\pi^N=u^N(\infty)\to \pi$, where $\Pi_{N}$ is the stationary distribution of $\displaystyle \{(U^{(N)}(t),t\geq 0\}$.
\end{theorem}

\begin{proof}
According to Theorem A.1 
in~\cite{DTZ}, the sequence of probability measures $(\Pi_N)_N$ is tight, since the state space is compact. Applying the  Theorem 2.2 (p.104) of \cite {EK}, we know that $(\Pi_N)_N$ is relatively compact and has limit points. Thus, we now only need to show that all convergent subsequences share the same limiting point, which is exactly $\delta_{\pi}$.

Let $(\Pi_{N_k})_k$ denote an arbitrary subsequence of $(\Pi_N)_N$, which converges to the distribution $\widetilde{\Pi}$.
Then, let the process $
\{ U^{N_k}(t),t\ge 0\}
$, $
\{U(t),t\ge 0\}
$ start from stationary state $\Pi_{N_k}$ and $\widetilde{\Pi}$, respectively. According to the Theorem in~\cite{DTZ} and Theorem \ref{gas}, it follows that $\{U^{N_k}(t)\}\longrightarrow\{U(t)\}$ in the sense of weak convergence for processes and $\Pi_{N_k}\longrightarrow\delta_{\pi}$. 
Thus, the theorem is proved.
\end{proof}

In what follows, we study the relationship of customers' strategic behaviours in the finite $N$ model and the corresponding mean field limit model.

\begin{theorem}
If $q^e$ is a Nash equilibrium of the dedicated customers for the mean field model in the partially observable case, then $q^e$ is a $\epsilon$-Nash equilibrium in the finite $N$ model for $N$ sufficiently large.
\end{theorem}

 \begin{proof}
According to Theorem \ref{main result}, $E(L^N(1))$ is bounded, since
\begin{equation*}
  E(L^N(1))\to E(L(1)),  \text{  as } N\to \infty,
\end{equation*}
where $E(L^N(1))$ and $E(L(1))$ are the mean queue length of the original $N$ system and the mean field model with arrival rate $\lambda$.

The mean sojourn time for a customer from dedicated arrival stream is bounded, since
\begin{equation*}
  E(W^N(q))=\frac{1}{\mu}\left(E(L^N(q))+1\right)\leq\frac{1}{\mu}\left(E(L^N(1))+1\right).
\end{equation*}
 Hence,
\begin{equation*}
  \sup\limits_{N}E|W^N|< \infty,
\end{equation*}
and $W^N$ is uniformly integrable. By the definition of uniform integrable and Theorem \ref{main result}, we obtain that for any $ 
\epsilon >0$, there exists an $
N_0$,  such that when $N>N_0$,
\begin{eqnarray}
 \nonumber \mid E(W^N(q))-E(W(q)) \mid\hspace{-3mm}&=&\hspace{-3mm}\mid \sum\limits_{k=0}\limits^{\infty}\frac{k+1}{\mu}\pi^N_k-\sum\limits_{k=0}\limits^{\infty}\frac{k+1}{\mu}\pi_k \mid\\
 \nonumber&\leq&\frac{\epsilon}{2C},
\end{eqnarray}
and then
\begin{equation*}
  \mid S^N(q)-S(q) \mid\leq\frac{\epsilon}{2},
\end{equation*}
where $E(W^N(q))$ and $S^N(q)$, ($E(W(q))$ and $S(q)$), are the mean sojourn time and utility for a customer from a dedicated arrival stream in the finite $N$ model (mean field model).

Then, it follows that
\begin{equation*}
  S^N(q^e)>S(q^e)-\frac{\epsilon}{2}>S(q)-\frac{\epsilon}{2}>S^N(q)-\epsilon, \text{ for all }q.
\end{equation*}
Hence, $q^e$ is an $\epsilon$-Nash equilibrium in the finite $N$ system.
\end{proof}

Similarly, we can conclude that $q^*$ is an approximation to the optimal social benefits strategy in the finite $N$ model.

\section{Numerical Analysis}\label{Sec:5}
In this section, the effects of the system parameters on the behaviours of customers and social benefits are discussed. We explain the practical meaning according to the numerical results.

Figures \ref{strategyvsmu}--\ref{strategyvsR} and Figures \ref{socvsmu}--\ref{socvsR} imply the impacts of parameters, $\mu$, $\lambda_s$, $R$, on the joining probabilities and the social benefits regarding equilibrium and optimal strategies, respectively. From the first three figures, it is obvious that $q^e\geq q^*$, which implies that independent customers are inclined to abuse the system. Meanwhile, one is considerate when optimizing social benefits. The optimal joining probabilities are always not greater than the equilibrium joining probabilities but the social benefits with the optimal joining probabilities are larger than that with equilibrium joining probabilities, which is clear from Figure \ref{socvsmu}--\ref{socvsR}.

\begin{figure}[htbp]
\centering
\includegraphics[scale=0.7]{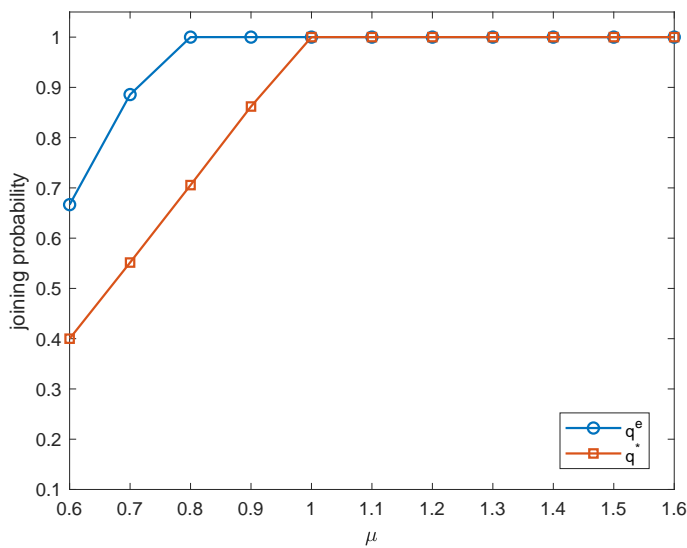}
\caption{Equilibrium and socially optimal joining probabilities vs. $\mu$ for $R=3,C=1.1,\lambda=0.5,\lambda_s=0.2$}\label{strategyvsmu}
\end{figure}
\begin{figure}[htbp]
\centering
 \includegraphics[scale=0.7]{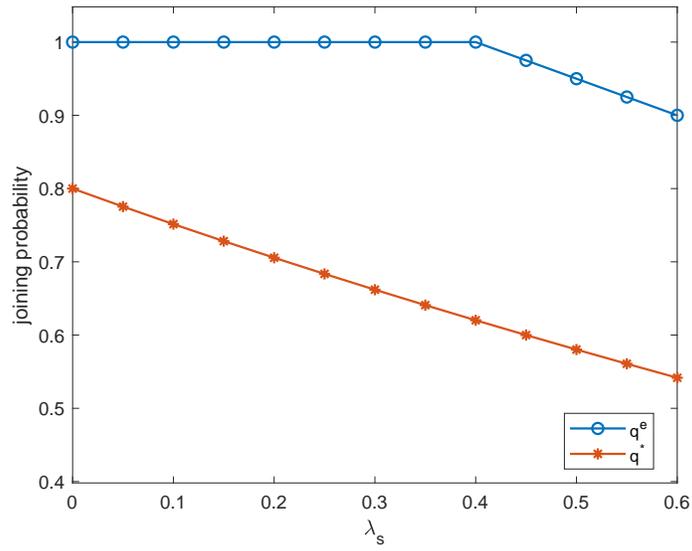}
\caption{Equilibrium and socially optimal joining probabilities vs. $\lambda_s$ for $R=3,C=1.1,\mu=1.3,\lambda=0.59$}\label{strategyvslambdas}
\end{figure}
\begin{figure}[htbp]
\centering
\includegraphics[scale=0.7]{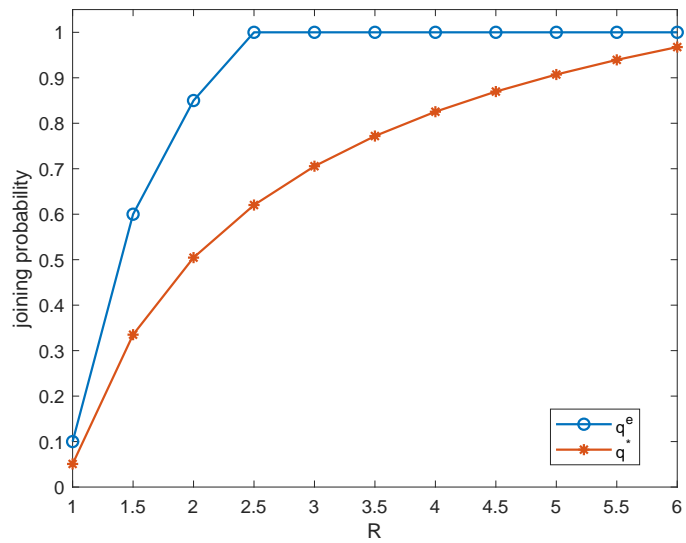}
\caption{Equilibrium and socially optimal joining probabilities vs. R for $C=2,\mu=1.3,\lambda=0.5,\lambda_s=0.5$}\label{strategyvsR}
\end{figure}

Figure \ref{strategyvsmu} depicts that the joining probabilities increase with the growth of service rate $\mu$, as expected, which reflects the fact that the larger the service rate is, the shorter the customer's expected waiting time is, both in maximizing individual or social benefits. In Figure \ref{strategyvslambdas}, both equilibrium and optimal joining probabilities are non-increasing with respect to  $\lambda_s$. It describes the fact that when arrival rates $\lambda$ and $\lambda_s$ increase, customers prefer to balk since they think that the system is more crowded and loaded. The result shown in Figure \ref{strategyvsR} is easy to understand: higher reward drives more customers to join the system.

\begin{figure}[htbp]
\centering
\includegraphics[scale=0.7]{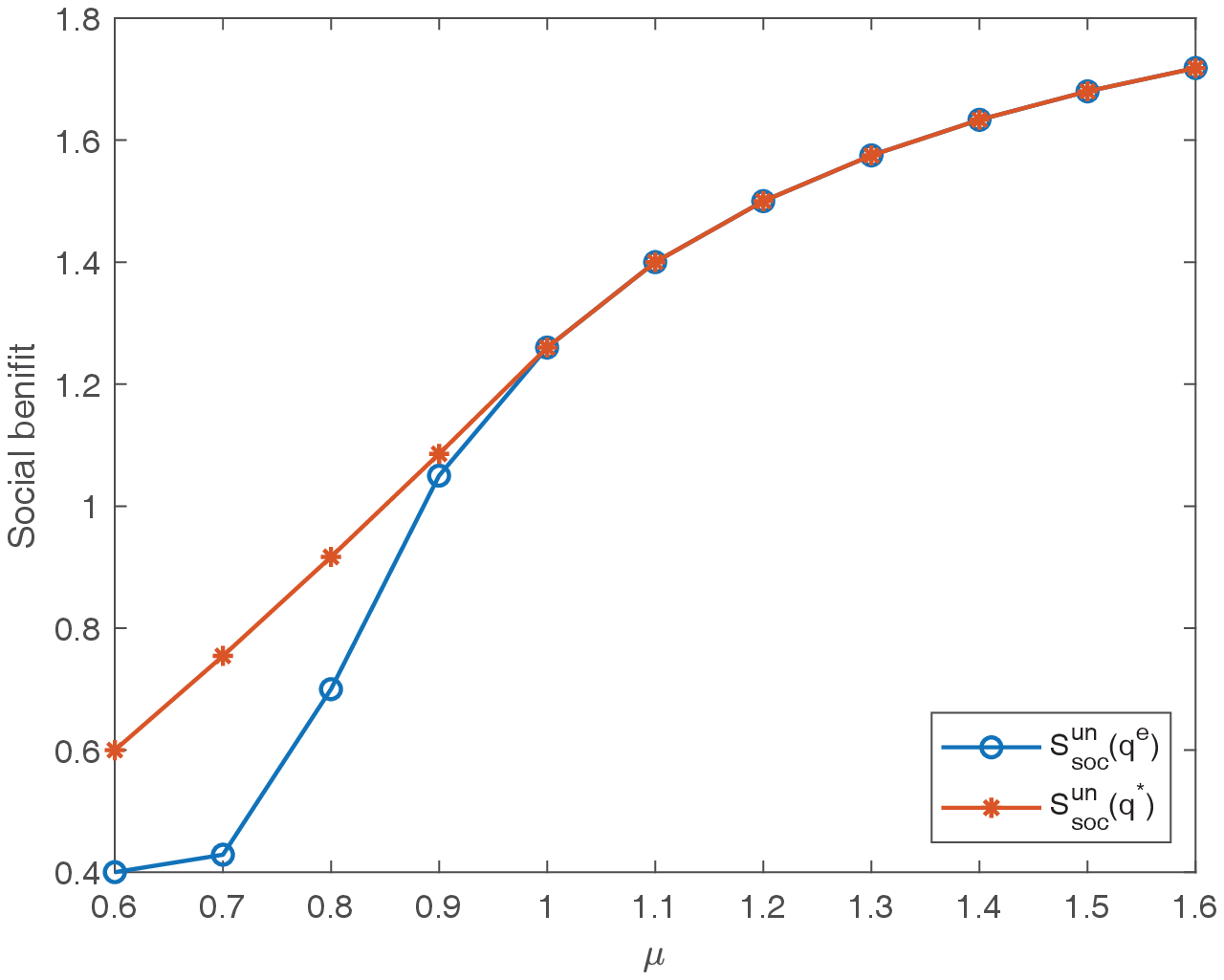}
\caption{Social benefits per time unit vs. $\mu$ for $R=3,C=1.1,\lambda=0.5,\lambda_s=0.2$}\label{socvsmu}
\end{figure}

\begin{figure}[htbp]
\centering
\includegraphics[scale=0.7]{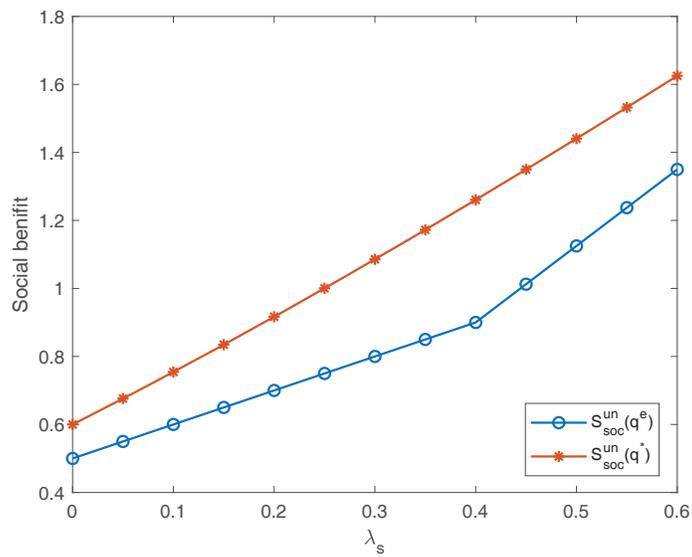}
\caption{Social benefits per time unit vs. $\lambda_s$ for $R=3,C=1.1,\mu=1.3,\lambda=0.59$}\label{socvslambdas}
\end{figure}

\begin{figure}[htbp]
\centering
 \includegraphics[scale=0.7]{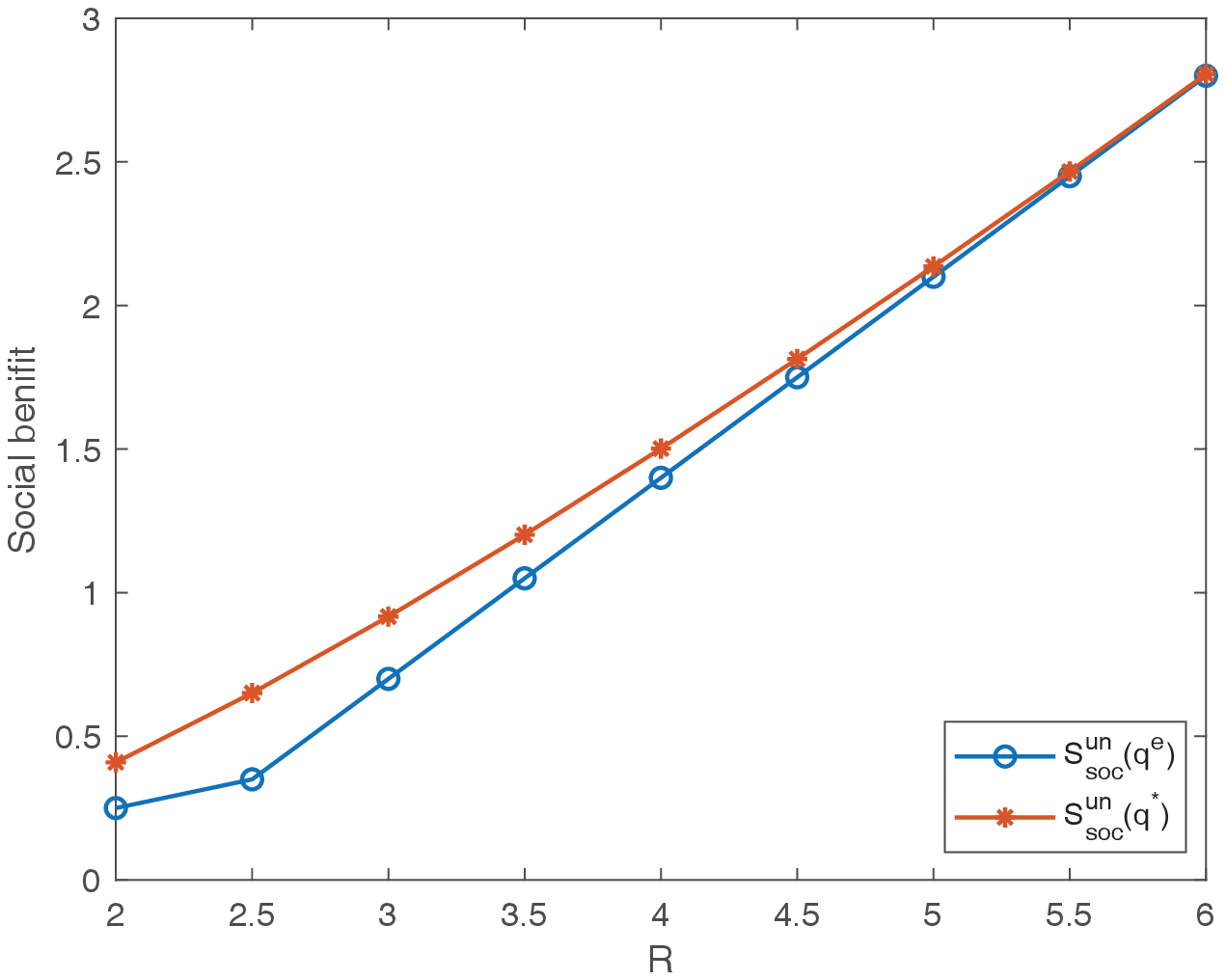}
\caption{Social benefits per time unit vs. R for $C=2,\mu=1.3,\lambda=0.5,\lambda_s=0.5$}\label{socvsR}
\end{figure}

Figures \ref{socvsmu} and \ref{socvsR} illustrate that the social benefits increase with respect to $\mu$ and $R$, respectively, when customers follow the social optimal joining strategy. This is true since higher service rate and more rewards are beneficial in obtaining larger social benefits.
The social benefits increases with increasing extra arrival rate $\lambda_s$ as shown in Figure \ref{socvslambdas}. This is true since the smart customer maximizes the benefits (its waiting time is zero).

\section{Conclusion}
We studied customers' strategic behaviours in a partially observable queueing system with $N$ nodes. In this queueing system, each node has a dedicated arrival stream and there is an extra arrival stream designed to balance the queue length. Using a mean field approximation, we first derived the equilibrium and optimal social strategies for a partially observable limit model. Then, we showed that the strategies obtained in the limit model can be used to approximate equilibrium strtegies for the finite $N$ model. We compared the joining probabilities and social benefits to the equilibrium and optimal strategies in numerical experiments and provided explanations for figures.
\section*{Acknowledgement}
This work was supported in part by the National Natural Science Foundation of China (No.61773014) and the Natural Sciences and Engineering Research Council of Canada.
We thanks the two anonymous reviewers for their constructive comments and remarks for improving the quality of the presentations of our work.

%
%

\end{document}